\documentclass[10pt,a4paper]{article}
\usepackage[dvips]{color}
\usepackage{indentfirst,latexsym,bm}
\usepackage{graphicx}
\usepackage{amsmath,amsthm}
\usepackage{amssymb}
\usepackage{multicol}
\usepackage{ulem}
\usepackage{amsfonts}
\usepackage{mathrsfs}
 \newtheorem{thm}{Theorem}[section]
 \newtheorem{cor}[thm]{Corollary}
 \newtheorem{lem}[thm]{Lemma}
 \newtheorem{prop}[thm]{Proposition}
 
 \theoremstyle{definition}
 \newtheorem{dfn}[thm]{Definition}
 \newtheorem{exmp}{Example}
 \newtheorem{nott}{Notation}
 \newtheorem{rem}{Remark}

  \DeclareMathAlphabet{\mathsfsl}{OT1}{cmss}{m}{sl}
  \newcommand{\calA}{\mathcal{A}}

\newcommand{\calJ}{\mathcal{J}}
\newcommand{\calL}{\mathcal{L}}
\newcommand{\IM}{\mathrm{Im}}
  \newcommand{\RE}{\mathrm{Re}}
 \newcommand{\Enum}{\mathbb{E}}
 \newcommand{\Rnum}{\mathbb{R}}
 \newcommand{\Cnum}{\mathbb{C}}
 
 \newcommand{\Nnum}{\mathbb{N}}
 \newcommand{\mi}{\mathrm{i}}
 \newcommand{\dif}{\mathrm{d}}
  \newcommand{\diag}{\mathrm{diag}}
 \newcommand{\tensor}[1]{\mathsf{#1}}
 \newcommand{\abs}[1]{\left\vert#1\right\vert}
 \newcommand{\set}[1]{\left\{#1\right\}}
 \newcommand{\norm}[1]{\left\Vert#1\right\Vert}
 \newcommand{\innp}[1]{\langle {#1}\rangle}

\title{On the eigenfunctions of the complex Ornstein-Uhlenbeck operators}
\author{\rm\small
\noindent CHEN Yong\\
\noindent \footnotesize School of Mathematics and Computing Science, Hunan
University of Science and Technology,\\
\noindent \footnotesize Xiangtan, Hunan, 411201,
P.R.China. chenyong77@gmail.com\\
\rm\small \noindent LIU Yong\\
\noindent \footnotesize LMAM, School of Mathematical Sciences, \
Center for Statistical Science,\\
\noindent \footnotesize Peking University, Beijing,  {\rm 100871}, P. R. China.
liuyong@math.pku.edu.cn (Corresponding author)\\
}
\date{}
\begin{document}
\maketitle
\maketitle \noindent {\bf Abstract } \\
Starting from the 1-dimensional complex-valued Ornstein-Uhlenbeck process, we present two natural ways to imply the associated eigenfunctions of the
2-dimensional normal Ornstein-Uhlenbeck operators in the complex Hilbert space $L_{\Cnum}^2(\mu)$. We call the eigenfunctions Hermite-Laguerre-Ito polynomials. In addition, the Mehler summation formula for the complex process are shown.
\maketitle


\section{Introduction}
There are abundant
literature in stochastic analysis concerning to the Hermite polynomials and the 1-dimensional real-valued Ornstein-Uhlenbeck process (\cite{guo,Ma,shg} and references therein),
\begin{equation*}
  \dif X_t=-bX_t\dif t+ \sqrt{2\sigma^2}\dif B_t.
\end{equation*}
If the processes $X_t,\,B_t$ and the coefficient $b$ are all complex-valued, what new findings can we get? Or say, what is the meaning of the 1-dimensional complex-valued Ornstein-Uhlenbeck process
\begin{equation}\label{cp}
  \dif Z_t=-\alpha Z_t\dif t+ \sqrt{2\sigma^2}\dif \zeta_t
\end{equation}
where $Z_t=X_1(t)+\mi X_2(t)$, $\alpha=ae^{\mi \theta}= r+\mi \Omega$, and $\zeta_t=B_1(t)+\mi B_2(t)$ is a complex Brownian motion?

First, it is clear that this complex-valued process can be represented by 2-dimensional nonsymmetric Ornstein-Uhlenbeck process
\begin{align}
\begin{bmatrix}
\dif X_1(t)\\ \dif X_2(t) \end{bmatrix}& =\begin{bmatrix} -a\cos \theta & a \sin \theta\\ -a \sin \theta & -a\cos \theta \end{bmatrix}
\begin{bmatrix} X_1(t)\\ X_2(t) \end{bmatrix} \dif t + \sqrt{2\sigma^2}
\begin{bmatrix} \dif B_1(t)\\ \dif B_2(t) \end{bmatrix}\nonumber \\
&=\begin{bmatrix} -r & \Omega\\ -\Omega & -r  \end{bmatrix}
\begin{bmatrix} X_1(t)\\ X_2(t) \end{bmatrix} \dif t + \sqrt{2\sigma^2}
\begin{bmatrix} \dif B_1(t)\\ \dif B_2(t) \end{bmatrix} .\label{langevin}
\end{align}
From view of the physics, this equation (\ref{langevin}) describes the motion of a charged test particle in present of a constant magnetic field and is
also called the A-Langevin equation in Ref.~\cite[p181-186]{bal}. For the process (\ref{langevin}), its generator is
\begin{equation}\label{ou.op}
A = (-rx+\Omega y)\frac{\partial}{\partial x} +
(-\Omega x-ry)\frac{\partial}{\partial y} +\sigma^2
(\frac{\partial^2}{\partial x^2}+\frac{\partial^2}{\partial y^2})\, ,
\end{equation}
and its stationary distribution is
\begin{equation}\label{nvmea}
\dif\mu = \frac{r}{2\pi \sigma^2
}\exp\set{-\frac{r(x^2+y^2)}{2\sigma^2}}\dif x\dif y.
\end{equation}
In the Hilbert space $L^2(\mu)$,  the adjoint of the operator $A$ is
\begin{equation}
A^* = (-rx-\Omega y)\frac{\partial}{\partial x} -
(-\Omega x+ry)\frac{\partial}{\partial y} +\sigma^2
(\frac{\partial^2}{\partial x^2}+\frac{\partial^2}{\partial y^2}).
\end{equation}
Thus, $A$ is a nonsymmetric operator and  it satisfies $A^* A=AA^*$ formally (e.g. it is valid for any polynomial ), and in fact, one can show that $A$ is a normal operator \cite{con}.

Now we address the first question, what is the spectrum and the associated eigenfunctions of the operator $A$.

The spectrum and the associated eigenfunctions of the symmetric Ornstein-Uhlenbeck operator (i.e., Hermite polynomials) are well-known \cite{guo}. Without the restriction of symmetric, the spectrum of any finite dimensional Ornstein-Uhlenbeck operator in $L^p(\mu)$ for $1< p<\infty$ is shown explicitly in \cite{Mg}. In details, let $\calL=\sum_{i,j=1}^N q_{ij}\frac{\partial^2}{\partial x_i\partial x_j}+\sum_{i,j=1}^N b_{ij}x_j \frac{\partial}{\partial x_i}$ be a possibly degenerate Ornstein-Uhlenbeck operator and assume that the associated Markov semigroup has an invariant measure $\mu$, then the spectrum of $\cal L$ is
 \begin{equation}\label{qqguai}
   \sigma(\calL)=\set{\gamma=\sum_{j=1}^r n_j \lambda_j:\,n_j\in \Nnum},
 \end{equation}
where $\lambda_1,\cdots,\lambda_r$ are the distinct eigenvalues of the matrix $(b_{ij})$ \cite[Theorem 3.1]{Mg}. This conclusion has been extended to the infinite dimensional case in some papers such as \cite{cmg,Krv,PE,Nvj}. The spectrum of the general non-symmetric Markovian semigroup is discussed in \cite{shg2}. However, for the nonsymmetric Ornstein-Uhlenbeck operator, the associated eigenfunctions are still unknown up to now.

It follows from (\ref{qqguai}) that the spectrum set of the operator $A$ in (\ref{ou.op}) in Hilbert space $L^2(\mu)$ is
 \begin{equation}\label{spectm}
   \sigma(A)=\set{-( m+n) r + \mi (m-n)\Omega, \,\,m,n=0,1,2,\dots}.
 \end{equation}

In the present note, we will show that eigenfunctions of the operator $A$
are expressed by the Hermite-Laguerre-Ito polynomials (see Definition~\ref{ijldf} and Theorem~\ref{dingl2}), which is called the Hermite polynomials of complex variables by K. Ito in \cite{ito}.  K. Ito used this polynomials to characterize the complex multiple Ito-Wiener integral, i.e., his aim was to show a relation similar to the well-known one between the real multiple Ito-Wiener integral and the Hermite polynomials \cite{guo}. His basic idea was that the Hermite polynomials of complex variables are the coefficients in the expansion of the generating function $\exp{(-t\overline{t}+t\overline{z}+\overline{t}z)}$ (see Theorem 12 in \cite{ito}). This is the complex-valued version of Hermite polynomials as the coefficients in the expansion of the generating function $\exp{(-t^2+2xt)}$.

The main purposes of this note are presenting two different ways to obtain the Hermite-Laguerre-Ito polynomials as eigenfunctions of the operator $A$. One way is that we use the direct and elementary computation by means of the normality of $A$ in $L^2(\mu)$. The other way is that by defining the complex variable creation operator and annihilation operator in the complex Hilbert space $L_{\Cnum}^2(\mu)$, we verify  that Hermite-Laguerre-Ito polynomials can be generated iteratively by the complex creation operator acting on the constant $1$ (see Definition~\ref{ijldf}). Those approaches give us some deeper and richer understanding of the 1-dimensional complex-valued Ornstein-Uhlenbeck processes and the ``nonsymmetric" stochastic analysis\cite{cmg1,Nvj1,shg1,shg2}. The concrete calculations are given in section 2.

 In section 3, as applications of the above computations, we obtain the  Mehler transform formula and the Mehler summation formula for the complex process (\ref{cp}) (or the process (\ref{langevin})). In section 4, a high-dimensional example which can be decomposed to the summation of series of 1-dimensional complex-valued Ornstein-Uhlenbeck processes is given. Finally, some tedious computations are listed in Appendix.

\section{The Hermite-Laguerre-Ito polynomials on $\Cnum$ (or say: $\Rnum^2$) }\label{sec2}
In this section, we present two way to imply the Hermite-Laguerre-Ito polynomials on $\Cnum$. One is by means of the normality of the operator, the other is by means of the creation operator and annihilation operator \cite{Ma}.

\subsection{From the view of the normal operator}\label{sec2.1}
Let $\rho=\frac{2\sigma^2}{r}$ and $\mi=\sqrt{-1}$. Denote by $H_n(x,\rho)$ the Hermite polynomials.  Let
\begin{equation}
\calA_s :=\frac{1}{2} (A +A^*), \quad \calJ :=\frac{1}{2\mi}
(A - A^*).
\end{equation}
Then we have that
\begin{eqnarray}\label{ou}
\calA_s = (\sigma^2\frac{\partial^2}{\partial x^2}-rx\frac{\partial}{\partial x})
+ (\sigma^2\frac{\partial^2}{\partial y^2}-ry\frac{\partial}{\partial y}),\quad
\calJ = -\mi\Omega (y\frac{\partial}{\partial x} -x\frac{\partial}{\partial y}).
\end{eqnarray}

 It is well known that  $\set{H_k(x,\frac{\rho}{2})H_{m+n-k}(y,\frac{\rho}{2}),\,0\leqslant k\leqslant m}$ is an orthogonal basis
 of $S_{m+n}$,  the characteristic subspace associated with the eigenvalue $-(m+n)r$ of $\calA_s$. Note that $A A^*= A^*A$ is valid for any polynomial.
  Thus $\calA_s \calJ= \calJ \calA_s$. Restriction of  $\calA_s, \calJ$ to the finite dimensional space $S_{m+n}$, the self-adjoint
  operators $\calA_s$ and $\calJ $ have common eigenfunctions which make up an orthogonal basis of $S_{m+n}$.
  It follows that each common eigenfunction is a linear combination of $H_k(x,\frac{\rho}{2})H_{m+n-k}(y,\frac{\rho}{2}),\,0 \leqslant k\leqslant m+n,$. Then we have the following:
\begin{prop}\label{pop1}
Let $l=m+n$ with $m,n\in \Nnum$ and $\vec{\beta}=(\beta_0,\,\beta_1,\,\beta_2,\dots,\beta_{l})'$, where
$\beta_k$ are complex numbers. If the function $J_{m,n}(x,y)=\sum\limits^{l}_{k=0}\beta_k H_k(x,\frac{\rho}{2})H_{l-k}(y,\frac{\rho}{2})$ satisfies
\begin{equation}\label{eigen.j0}
   \calJ J_{m,n}(x,y)= -\mi \lambda \Omega J_{m,n}(x,y),
\end{equation}
then the linear equation
\begin{equation}\label{linea}
\tensor{M}(\lambda)\vec{\beta}=0
\end{equation}
holds, where $\tensor{M}(\lambda)$ is an $(l+1)\times (l+1)$ tridiagonal matrix
  \begin{equation}\label{matrixM}
\tensor{M}(\lambda)= \left[
\begin{array}{llllllll}
     -\lambda & 1 &0    & 0   & \dots  & 0 & 0& 0\\
     -l & -\lambda &2  & 0   & \dots  & 0 & 0& 0\\
     0   & 1-l & -\lambda &3  & \dots  & 0 & 0& 0\\
     \hdotsfor{8}\\
     0   & 0   & 0   &0    & \dots  &-2 &-\lambda  & l\\
     0   & 0   & 0   &0    & \dots  &0 &-1  & -\lambda
\end{array}
\right ].
\end{equation}
When $\lambda=-(m-n)\mi$,
 Eq.(\ref{linea}) has a nonzero solution, and the solution $\vec{\beta}$ satisfies
 \begin{equation}\label{itheta2}
\sum\limits^l_{k=0}\beta_{k} x^{k} y^{l-k}=(x+\mi y)^{m-n}(x^2+y^2)^n c = (x+\mi y)^{m}(x-\mi y)^n c,\quad x,y\in \Rnum,
\end{equation}
where $c$ is a constant.
\end{prop}
\begin{thm}\label{dingl2}
The eigenfunction associated with the eigenvalue $-r (m+n) - \mi (m-n)\Omega$ of the OU operator $A$ is
  \begin{equation*}
   J_{m,n}(x,y)= \left\{
      \begin{array}{ll}
      (-1)^n{n!} (x + \mi y)^{m-n}L^{m-n}_{n}(x^2+y^2,\rho),\quad m\ge n,  \\
      (-1)^{m}m!(x - \mi y)^{n-m}L^{n-m}_{m}(x^2+y^2,\rho),\quad m<n,
      \end{array}
\right. \label{hxy2}
      \end{equation*}
where $L_n^{\alpha}(x,\rho)$ is the Laguerre polynomials (see Definition~\ref{lgldf} or References~\cite{ch,guo,Le,wg}).
\end{thm}
Proof of Proposition~\ref{pop1} and Theorem~\ref{dingl2} are presented in Section~\ref{ssec5.1}.

\subsection{From the view of the creation operator and annihilation operator }
An elegant method to deduce the Hermite polynomials is given by means of the creation operator and annihilation operator \cite{Ma}. Let $x\in \Rnum$ and $\gamma(\dif x)=\frac{1}{\sqrt{2\pi\sigma}}e^{-x^2/2\sigma}\dif x$. The operator $\partial$ is the operator of differentiation(or say: the annihilation operator): $$\partial \varphi(x)=\varphi'(x),\quad \varphi(x)\in C^1(\Rnum).$$ Its adjoint operator in the Hilbert space $L^2(\gamma)$ is $$\partial^*\varphi(x)=-\varphi'(x)+\frac{x}{\sigma}\varphi(x).$$ Then the Hermite polynomials are defined as the sequence\footnote{In \cite{Ma}, the variance is $\sigma=1$.}
\begin{align*}
  H_0(x)&=1,\\
  H_n(x)&=\sigma \partial^* H_{n-1}(x)=\sigma^n (\partial^* )^n 1.
\end{align*}

Let $z=x+\mi y$ and $\mu$ be as Eq.(\ref{nvmea}).
We consider the complex Hilbert space $L_{\Cnum}^2(\mu)$ associated to the inner product $$\innp{f,g}=\int_{\Rnum^2} f\bar{g} \mu(\dif x \dif y).$$
Set $C^1_0(\Rnum^2)$ the collection of $C^1$-function with compact support. We denote by $\partial,\,\bar{\partial} $ the operators of differentiation (or say: the complex annihilation operator)
\begin{align*}
  (\partial \phi)(z)= \frac{\partial}{\partial z}\phi(z) =\frac{1}{2} (\frac{\partial}{\partial x}-\mi \frac{\partial}{\partial y})\phi(x,y),\quad
  (\bar{\partial} \phi)(z)=\frac{\partial}{\partial \bar{z}}\phi(z) =\frac{1}{2} (\frac{\partial}{\partial x}+\mi \frac{\partial}{\partial y})\phi(x,y).
\end{align*}
By Lemma~2.1 of \cite[p6]{Ma}, the direct calculating yields the adjoint of the operators $\partial,\,\bar{\partial} $ as follows.
\begin{lem}[complex creation operator]\label{lmm1}
Denote by $\partial^*,\,\bar{\partial}^*$ the operator defined, for $\phi\in C^1_0(\Rnum^2)$, by
  \begin{equation*}
    (\partial^*\phi)(z)=-\frac{\partial}{\partial \bar{z}}\phi(z)+\frac{z }{\rho}\phi(z),\quad (\bar{\partial}^*\phi)(z)=-\frac{\partial}{\partial {z}}\phi(z)+\frac{\bar{z}}{\rho }\phi(z).
  \end{equation*}
 Then if $\partial \phi$ and $\partial^*\psi\in L_{\Cnum}^2(\mu)$ we have
 \begin{equation}\label{adjt1}
   \innp{\partial\phi,\,\psi}=\innp{\phi,\,\partial^*\psi},\quad \innp{\bar{\partial}\phi,\,\psi}=\innp{\phi,\,\bar{\partial}^*\psi}.
 \end{equation}
\end{lem}

Clearly, $\partial$ commutes with $\bar{\partial}$ and $\partial^* $ with $\bar{ \partial}^*$.
\begin{dfn}[Definition of the Hermite-Laguerre-Ito polynomials]\label{ijldf}
Let $m,n\in \Nnum$. We define the sequence on $\Cnum$ (or say: $\Rnum^2$)
\begin{align}
  J_{0,0}(z,\rho)&=1,\nonumber\\
  J_{m,n}(z,\rho)&={\rho}^{m+n}(\partial^*)^m(\bar{\partial}^*)^n 1.\label{itldefn}
\end{align}
We call it the Hermite-Laguerre-Ito polynomials in the present paper.
\end{dfn}
By induction on $m,n$, we see that $J_{m,n}$ is a polynomial of degree $m+n$ and its term of highest degree is $z^m\bar{z}^n$.
The first few Hermite-Laguerre-Ito polynomials are
\begin{align*}
  J_{m,0}&=z^m,\quad J_{0,n}=\bar{z}^n,\\
  J_{1,1}&= \abs{z}^2-\rho,\quad J_{2,1}=z(\abs{z}^2-2\rho),\quad J_{3,1}=z^2(\abs{z}^2-3\rho),\dots\\
  J_{1,2}&= \bar{z}(\abs{z}^2-2\rho),\quad J_{2,2}=\abs{z}^4-4\rho\abs{z}^2+2\rho^2,\quad J_{3,2}=z(\abs{z}^4-6\rho\abs{z}^2+6\rho^2),\dots\\
  \dots
\end{align*}
In general, we have that
\begin{thm}\label{ijl}
The Hermite-Laguerre-Ito polynomials satisfy
\begin{equation}\label{jal}
  J_{m,n}(z,\,\rho)=\sum_{r=0}^{m\wedge n}(-1)^r r!{m \choose r}{n \choose r}z^{m-r}\bar{z}^{n-r}\rho^{r}.
\end{equation}
\end{thm}
Proof of Theorem~\ref{ijl} are presented in Section~\ref{ssec5.2}.
\begin{rem}
  If $\rho=1$ then $J_{m,n}(z,\,1)$
 is called the Hermite polynomials of complex variables by K. Ito, who has shown that a close relation between the polynomials and the complex multiple Ito-Wiener integral \cite{ito}, the reader can also refer to Section~\ref{ssec5.3}. If $\zeta$ is a symmetric complex Gaussian variable,  by calculating Feynman diagrams (see \cite{jan}, p131),  Wick product is represented by
\begin{equation*}
  :\zeta^{m}\bar{\zeta}^n:=J_{m,n}(\zeta,\,\Enum{\abs{\zeta}}^2).
\end{equation*}
\end{rem}

From the above power series expression, we get that
 \begin{align}
  J_{m,n}(z,\,\rho)&= \left\{
      \begin{array}{ll}
      z^{m-n}\sum\limits_{r=0}^n(-1)^r r!{m \choose r}{n \choose r}\abs{z}^{2(n-r)}\rho^r, & m\ge n,\\
      \bar{z}^{n-m}\sum\limits_{r=0}^m(-1)^r r!{m \choose r}{n \choose r}\abs{z}^{2(m-r)}\rho^r & m<n, \\
      \end{array}\nonumber \right.\\
  &=  \left\{
      \begin{array}{ll}
      z^{m-n}(-1)^nn!L^{m-n}_n (\abs{z}^2,\,\rho), & m\ge n,\\
      \bar{z}^{n-m}(-1)^mm!L^{n-m}_m (\abs{z}^2,\,\rho), & m<n. \\
      \end{array}\right.\label{jl}
\end{align}
 Thus, we name $J_{m,n}(z,\,\rho)$ as the Hermite-Laguerre-Ito polynomials in the present paper.

\begin{thm} \label{thjl}
The Hermite-Laguerre-Ito polynomials satisfy the following:
  \begin{itemize}
\item[\textup{1)}] Orthonormal basis: $\set{(m!n!\rho^{m+n})^{-\frac12}J_{m,n}(z,\,\rho):\,m,n\in \Nnum}$ is an orthonormal basis of $L_{\Cnum}^2(\mu)$.
\item[\textup{2)}] Eigenfunctions: Let $c\in \Rnum$,
\begin{equation}\label{eigen}
  [(1+\mi c)z \frac{\partial}{\partial z}+(1-\mi c)\bar{z} \frac{\partial}{\partial \bar{z}}-2\rho \frac{\partial^2}{\partial z\partial \bar{z}} ]J_{m,n}(z,\,\rho)=[m+n+\mi (m-n)c]J_{m,n}(z,\,\rho).
 \end{equation}
\item[\textup{3)}] Generating function: Let $\lambda\in \Cnum$,
  \begin{equation}\label{gene}
    \exp\set{\lambda \bar{z} + \bar{\lambda}z-\rho |\lambda|^2}=\sum_{m=0}^{\infty}\sum_{n=0}^{\infty}\frac{\bar{\lambda}^m\lambda^n}{m!n!}J_{m,n}(z,\rho).
  \end{equation}
 \end{itemize}
 \end{thm}
Proof of Theorem~\ref{thjl} are presented in Section~\ref{ssec5.2}.

Let $\rho=\frac{2\sigma^2}{r},\,c=\frac{\Omega}{r}$. It is well known that $\frac{\partial}{\partial z}=\frac12 (\frac{\partial}{\partial x}-\mi \frac{\partial}{\partial y}),\, \frac{\partial}{\partial \bar{z}}=\frac12 (\frac{\partial}{\partial x}+\mi \frac{\partial}{\partial y}),\,4 \frac{\partial^2}{\partial z\partial \bar{z}}=\frac{\partial^2}{\partial x^2}+\frac{\partial^2}{\partial y^2}$, thus
 \begin{align}
 {A} &=(-rx+\Omega y)\frac{\partial}{\partial x} +
 (-\Omega x-ry)\frac{\partial}{\partial y} +\sigma^2 (\frac{\partial^2}{\partial x^2}+\frac{\partial^2}{\partial y^2})\nonumber\\
 &=-r[(1+\mi c)z \frac{\partial}{\partial z}+(1-\mi c)\bar{z} \frac{\partial}{\partial \bar{z}}-2\rho\frac{\partial^2}{\partial z\partial \bar{z}}]. \label{eqeq}
 \end{align}
Then it follows from Theorem~\ref{thjl} that we have
\begin{thm}
Let $\rho=\frac{2\sigma^2}{r}$. The Hermite-Laguerre-Ito polynomials $J_{m,n}(z,\rho)$ are the eigenfunctions of 2-dimensional normal Ornstein-Uhlenbeck operators $A$ on $L_{\Cnum}^2(\mu)$ with respect to the eigenvalue $-(m+n)r-\mi (m-n)\Omega$, $m,n\ge 0$. And the operator ${A}$ on the Hilbert space $L_{\Cnum}^2(\mu)$
has a pure point spectrum.\footnote{The conclusion of pure point spectrum has been known in \cite{Mg} by a different way.}
\end{thm}

These eigenfunctions $J_{m,n}(z,\rho)$ can be employed for, say, orthogonal
decomposition like the Hermite polynomials.
\begin{cor}
  Every function f in $L_{\Cnum}^2(\mu)$ has a unique series expression
  \begin{equation*}
    f(x,y)=\sum_{m=0}^{\infty}\sum_{n=0}^{\infty}a_{m,n}\frac{J_{m,n}(z,\rho)}{{m!n!\rho^{m+n}}},
  \end{equation*}
  where the coefficients $a_{m,n}$ are given by $ a_{m,n}=\innp{f,\,J_{m,n}}$.
  Moreover, we have \begin{equation*}
    \norm{f}^2=\sum_{m=0}^{\infty}\sum_{n=0}^{\infty}\frac{|a_{m,n}|^2}{m!n!\rho^{m+n}}.
  \end{equation*}
\end{cor}
The following are two examples.
\begin{align*}
  z^m\bar{z}^n&= \sum_{k=0}^{m\wedge n} {m \choose k}{n \choose k}k!\rho^{k} J_{m-k,n-k}(z,\rho), \nonumber\\
    H_k(x,\frac{\rho}{2})H_{l-k}(y,\frac{\rho}{2})& = \frac{\mi^{l-k}}{2^{l}}\sum\limits_{m=0}^l \sum_{u+v=m}{k \choose u}{l-k \choose v}(-1)^{v} J_{m,l-m}(z,\rho).
   \end{align*}
Clearly, the last equation displayed is equivalent to
   \begin{equation*}
J_{m,l-m}(z,\rho) = \sum\limits_{k=0}^{l}{\mi^{l-k}}\sum_{u+v=k}{m \choose u}{l-m \choose v}(-1)^{l-m-v} H_k(x,\frac{\rho}{2})H_{l-k}(y,\frac{\rho}{2}),
   \end{equation*}
   which is exact the expression of $J_{m,n}$ given in Proposition~\ref{pop1}.
\section{The normal Ornstein-Uhlenbeck semigroup and its Mehler summation formula }\label{sec3}
As an application to Section~\ref{sec2}, we will show the Mehler summation formula for the 2-dimensional normal Ornstein-Uhlenbeck semigroup in this section. This is an analogue of the Mehler summation formula for the real 1-dimensional Ornstein-Uhlenbeck process \cite[p51]{jan}, i.e., if $\gamma(\dif y)=\frac{1}{\sqrt{2\pi\rho} }e^{-\frac{y^2}{2\rho}}$, we have that
\begin{align*}
  \mathcal{M}_u\varphi (x)&= \int _{\Rnum}\,\varphi(y)\frac{1}{\sqrt{1-u^2} }\exp\set{-\frac{u^2{y}^2+u^2{x}^2-2uxy}{2\rho(1-u^2)}}\gamma(\dif y),\\
  \sum_{n=0}^{\infty}\frac{u^n}{n!\rho^n}H_n(x,\rho)H_n(y,\rho)&=\frac{1}{\sqrt{1-u^2} }\exp\set{-\frac{u^2{y}^2+u^2{x}^2-2uxy}{2\rho(1-u^2)}},
\end{align*}
where the sum converges pointwise and in $L^2(\gamma)\otimes L^2(\gamma)$, and $\mathcal{M}_u$ is known as the Mehler transform.

Now let $x,y,z\in \Rnum^2$. $A$ is the same as Eq.(\ref{ou}).
The operator $P_t=e^{t A},t\ge 0$ forms an operator semigroup known as the Ornstein-Uhlenbeck semigroup.
It is well known that the semigroup $P_t$ has the following explicit representation \cite{Mg}, due to Kolmogorov:
\begin{equation}\label{eq31}
  P_t \varphi (x)=\frac{1}{\pi\rho(1-e^{-2rt})}\int_{\Rnum^2}\,e^{-\frac{1}{\rho(1-e^{-2rt})}\abs{y}^2}\varphi(e^{tB}x-y)\,\dif y,
\end{equation}
where $B=\begin{bmatrix} -r & \Omega\\ -\Omega & -r  \end{bmatrix}$.
Clearly, $e^{tB}=e^{-rt}\begin{bmatrix} \cos \Omega t & \sin \Omega t\\ -\sin \Omega t & \cos \Omega t  \end{bmatrix}:=e^{-rt}B_0(t) $. Thus, a change of variable yields that the  normal  Mehler formula is
\begin{eqnarray*}
  P_t \varphi (x)&=&\int_{\Rnum^2}\,\varphi(e^{-rt}B_0(t) x+\sqrt{1-e^{-2rt}}z)\, \mu(\dif z)\nonumber \\
  &=& \int_{\Rnum^2}\,\varphi(B_0(t) x\cos\theta +z\sin \theta)\, \mu(\dif z)\\
  &=&P^s_t\varphi( B_0(t) x),\label{gxi}
\end{eqnarray*}
where $\mu$ is as Eq.(\ref{nvmea}), $\cos \theta=e^{-rt}$ with $\theta\in (0,\frac{\pi}{2})$, and $P^s_t$ is the 2-dimensional symmetric Ornstein-Uhlenbeck semigroup associated with the generator $\mathcal{A}_s$. The above equation presents a relation between the symmetric and nonsymmetric (normal) Ornstein-Uhlenbeck semigroup.

Denote $u=e^{-rt},\,\tilde{B}_0(u)=B_0(-\frac{1 }{r}\log u)$. A change of variable of Eq. (\ref{eq31}) yields
\begin{equation}\label{k2}
  \mathcal{M}_u\varphi (x)=P_t \varphi (x)=\int _{\Rnum^2}\,\varphi(y)\frac{1}{1-u^2 }\exp\set{-\frac{u^2\abs{y}^2+u^2\abs{x}^2-2u(\tilde{B}_0(u)x,y)}{\rho(1-u^2)}}\mu(\dif y),
\end{equation}
where $(x,\,y)=x'y$.

Eq.(\ref{eigen}) and Eq.(\ref{eqeq}) imply that the Mehler transform is characterized by \begin{equation*}
   \mathcal{M}_u J_{m,n}=u^{m+n+\mi(m-n)c}J_{m,n}.
 \end{equation*} Since the collection $\set{J_{m,n}}$ is an orthogonal basis in $L_{\Cnum}^2(\mu)$ with $\norm{J_{m,n}}^2=m!n!\rho^{m+n}$, this means that we also have
\begin{equation} \label{k1}
  \mathcal{M}_u\varphi (x)=\int _{\Rnum^2}\,\sum_{m=0}^\infty\sum_{n=0}^{\infty} \frac{u^{m+n+\mi(m-n)c}}{m!n!\rho^{m+n}}J_{m,n}(x)J_{m,n}(y)\,\varphi(y)\mu(\dif y),\quad \varphi\in L_{\Cnum}^2(\mu),
\end{equation}
with the sum $ \sum \sum \frac{u^{m+n+\mi(m-n)c}}{m!n!\rho^{m+n}}J_{m,n}(x)J_{m,n}(y)$ converging in $L_{\Cnum}^2(\mu)\otimes L_{\Cnum}^2(\mu)$ by the Riesz-Fischer Theorem.
Since the kernels in (\ref{k2}) and (\ref{k1}) have to coincide, we have that
\begin{thm}[Mehler summation formula]
  \begin{equation*}\label{k23}
  \sum_{m=0}^\infty\sum_{n=0}^{\infty}\frac{u^{m+n+\mi(m-n)c}}{m!n!\rho^{m+n}}J_{m,n}(x)J_{m,n}(y)=\frac{1}{1-u^2 }\exp\set{-\frac{u^2\abs{y}^2+u^2\abs{x}^2-2u(\tilde{B}_0(u)x,y)}{\rho(1-u^2)}},
\end{equation*}
where the sum converges pointwise and in $L_{\Cnum}^2(\mu)\otimes L_{\Cnum}^2(\mu)$.
\end{thm}
The pointwise convergent can be showed by the same argument as in the proof of Eq.(\ref{gene}).
\section{The processes are decomposed of some 1-dimensional complex-valued processes}\label{sec31}
In this section, suppose that $\tensor{C}$ is a normal $n$-by-$n$ matrix and the Ornstein-Uhlenbeck processes satisfy
\begin{equation*}
  \dif \vec{X}(t) = \tensor{C}\vec{X}(t)\dif t +  \sqrt{2\sigma^2} \dif \vec{B}(t).
\end{equation*}
Its generator is
\begin{equation}\label{gene3}
   {\mathcal{A}} = \sigma^2 \Delta_ {\vec{x}} + (\tensor{C}\vec{x})\cdot \nabla_ {\vec{x}}.
\end{equation}
Clearly, it is an $n$-dimensional normal Ornstein-Uhlenbeck operator.

As in \cite{lun,Mg}, we do linear transform according to the canonical form of the normal matrix $\tensor{C}$.
Since $\tensor{C}$ is normal, there is a real orthogonal matrix $\tensor{Q}$ such that
\begin{equation*}
 \tensor{Q}'\tensor{C}\tensor{Q}=\diag\set{A_1,\,A_2,\,\cdots,A_l },
\end{equation*}
where each $ A_j$ is either a real 1-by-1 matrix or is a real 2-by-2 matrix of the form
\begin{equation*}
   A_j= \begin{bmatrix}   \alpha_j & \beta_j \\ - \beta_j & \alpha_j  \end{bmatrix},
\end{equation*}
where $\alpha_j \pm \mi \beta_j,\,\beta_j\neq 0$ is a pair of conjugate complex eigenvalues of $\tensor{C}$ \cite[p105]{hj}. That is to say, its canonical form is a block diagonal matrix. If $\tensor{C}$ is not symmetric then there is at least one $ A_j$ which is a 2-by-2 matrix. Let $\vec{Y}(t)=\tensor{Q}' \vec{X}_t,\,\vec{y}=\tensor{Q}' \vec{x}$, then $\vec{Y}(t)$ has that
\begin{equation*}
  \dif \vec{Y}(t) = \diag\set{A_1,\,A_2,\,\cdots,A_l }\vec{Y}(t)\dif t +  \sqrt{2\sigma^2} \dif \vec{B}(t),
\end{equation*}
which is called the canonical form of OU processes in the present paper.
Its generator is
\begin{equation*}
 \tilde{\mathcal{A}}=\sigma^2 \Delta_{ \vec{y}}+ (\diag\set{A_1,\,A_2,\,\cdots,A_l } \vec{y})\cdot \nabla_ {\vec{y}}.
 \end{equation*}

The eigenfunctions of the 2-dimension normal operator (\ref{ou.op}) yield the eigenfunctions of the operator $\tilde{\mathcal{A}}$ \cite[p51]{RdSm}. Note that the eigenfunctions of the operator $ {\mathcal{A}} $ and $\tilde{\mathcal{A}}$ are same up to an orthogonal transform (i.e.,  $ \mathcal{A}f(\vec{x})=\tilde{\mathcal{A}}g(\vec{y})$ when $f(\vec{x})=g(\tensor{Q}'\vec{x})$), thus we can get the eigenfunctions of the operator $ {\mathcal{A}} $.
\begin{exmp}
The following equation describes the coupling diffusion on the lattice of the circle.
\begin{equation*}
  \dif \vec{X}(t)=\tensor{C}  \vec{X}(t)dt-r  \vec{X}(t)dt+\sqrt{2\sigma^2} \dif \vec{B}(t),
\end{equation*}
where $\tensor{C}$ is an $n$-by-$n$ ($n\ge 3$) tridiagonal matrix
\begin{eqnarray*}
\tensor{C}=\left[
\begin{array}{lllllc}
-(a+b) & a  & 0 &\cdots & b\\
     b  & -(a+b) & a &\cdots & 0\\
     \vdots &\vdots &\vdots &\vdots &\vdots\\
     0  & 0  & \cdots & -(a+b) & a\\
     a  & 0  & \cdots & b  & -(a+b) \
\end{array}
\right ].
\end{eqnarray*}
\end{exmp}
Denote by $\omega_j=\exp{(\mi\frac{2j\pi}{n})}$ the $j$-th unit root of $1$,
and $\omega^*_j$ is the complex conjugate of $\omega_j$. Clearly, the eigenvectors of $\tensor{C}$ are
$$\vec{\varphi}_{j}=\frac{1}{\sqrt{n}}(1,\omega_j,\,\omega_j^2,\cdots, \omega_j^{n-1})',\qquad j=0,1,2,\dots,n-1$$
and $\tensor{C}$ is a normal matrix.
 Let $L_k={\rm span}\set{\RE(\vec{\varphi}_{k}),\IM(\vec{\varphi}_{k})},\quad k=0, 1,2,\dots, [\frac{n}{2}] $ and denote by $P_{L_k}$ the project operator. Then there is a resolution of the identity $Id=\sum^{[\frac{n}{2}]}_{k=0} P_{L_k}.$
Thus
\begin{equation*}
   \vec{X}(t)=\sum^{[\frac{n}{2}]}_{k=0} P_{L_k}\vec{X}(t) =\sum^{[\frac{n}{2}]}_{k=0}( Y_k(t)\RE(\vec{\varphi}_{k}) +Z_k(t)\IM(\vec{\varphi}_{k})),
\end{equation*}
where $(Y_k(t),\,Z_k(t))$ satisfies the equation
\begin{align*}
\begin{bmatrix}
\dif Y_k(t)\\ \dif Z_k(t) \end{bmatrix}& =\begin{bmatrix} -r +  (a+b)(\cos\frac{2k\pi}{n}-1) & (a-b)\sin\frac{2k\pi}{n} \\ (b-a)\sin\frac{2k\pi}{n} & -r +(a+b)(\cos\frac{2k\pi}{n}-1) \end{bmatrix}
\begin{bmatrix} Y_k(t)\\ Z_k(t) \end{bmatrix} \dif t + \sqrt{2\sigma^2}
\begin{bmatrix} \dif B_1(t)\\ \dif B_2(t) \end{bmatrix} .
\end{align*}


\section{Appendix}
\subsection{Hermite polynomials and Laguerre polynomials}
To self-contained, we list the well-known results of the Hermite polynomials and the Laguerre polynomials, for which the reader can refer to the References~\cite{ch,guo,Le,wg}.
\begin{dfn}
  The Hermite polynomials are defined by the formula  $$H_n(x,\rho)=(-\rho)^n e^{x^2/2\rho}\frac{\dif^n}{\dif x^n}e^{-x^2/2\rho},\,n=1,2,\dots.$$
Clearly, it has power series expression,
\begin{equation*}
  H_n(x,\,\rho)= \sum\limits^{[n/2]}_{k=0}{n \choose 2k}(2k-1)!! x^{n-k}(-\rho)^k.
\end{equation*}
For simplicity of the symbols, we also denote it by $H_n(x)$.
\end{dfn}
\begin{prop}\label{lemm0}
 The Hermite polynomials $H_n(x,\rho)$ satisfy:
 \begin{itemize}
     \item[\textup{1)}] Partial derivatives: $
       \frac{\dif}{\dif x}H_n(x,\rho)=n H_{n-1}(x,\rho).$
     \item[\textup{2)}] Recursion formula:
$H_{n+1}(x,\rho)=x H_n(x,\rho)-n\rho H_{n-1}(x,\rho).$
 \item[\textup{3)}] Orthogonality: $\frac{1}{\sqrt{2\pi \rho}}\int_{-\infty}^{+\infty} e^{-\frac{x^2}{2\rho}}H_n(x,\rho)H_m(x,\,\rho)\,\mathrm{d}x=n!\rho^{n}\delta_{nm}.
$
     \item[\textup{4)}] Integral representation:
     \begin{equation}\label{hnin}
  H_n(x,\rho)=\frac{(-\mi )^n}{\sqrt{2\pi\rho}}\int_{-\infty}^{\infty} t^n\,e^{-\frac{1}{2\rho}(t-\mi x)^2}\dif t,\quad n=0,1,2,\cdots.
\end{equation}
     \end{itemize}
 \end{prop}
\begin{dfn}[Bessel's integrals]
 The definition of the Bessel function of the first kind, for integer values of $n$, is possible using an integral representation:
 \begin{equation}\label{bess}
  \mathrm{J}_n (x)  = \frac{1}{2 \pi} \int_{-\pi}^\pi e^{-\mathrm{i}\,(n \tau - x \sin \tau)} \,\mathrm{d}\tau.
 \end{equation}
\end{dfn}
\begin{dfn}\label{lgldf}
  The Laguerre polynomials $L_n^{\alpha}(x,\,\rho)$ is defined by the formula \cite {Le}
  \begin{equation}
 L_n^{\alpha}(x,\,\rho)=\frac{\rho^n}{ n!}x^{-\alpha} e^{\frac{x}{\rho}} {\dif ^n \over \dif x^n} \left(e^{-\frac{x}{\rho}} x^{n+\alpha}\right),\quad n=1,2,\dots
  \end{equation}
for arbitrary real $\alpha>-1$. Clearly, it has power series expression,
\begin{equation}\label{powersr}
  L_n^{\alpha}(x,\,\rho)= \frac{(-1)^n}{n!}\sum\limits^{n}_{r=0}(-1)^r r!{n+\alpha \choose r}{n \choose r}x^{n-r}\rho^r.
\end{equation}
\end{dfn}
\begin{prop}\label{lagu}
   The Laguerre polynomials $L_n^{\alpha}(x,\,\rho)$ satisfy:
     \begin{itemize}
     \item[\textup{1)}] Partial derivatives: $ \frac{\partial }{\partial x}L_n^{\alpha}(x,\,\rho)= -L_{n-1}^{\alpha+1}(x,\,\rho).      $
     \item[\textup{2)}] Recursion formula:
     \begin{eqnarray}
       L_n^{\alpha}(x,\,\rho)&=&L_n^{\alpha+1}(x,\,\rho)-\rho L_{n-1}^{\alpha+1}(x,\,\rho),\label{recur1}\\
       x \frac{\partial }{\partial x}L_n^{\alpha}(x,\,\rho)&=&n L_n^{\alpha}(x,\,\rho)-(n+\alpha)\rho L_{n-1}^{\alpha}(x,\,\rho).\label{recur2}
     \end{eqnarray}
     \item[\textup{3)}] Orthogonality: $\int_0^{+\infty}x^\alpha e^{-\frac{x}{\rho}}L^\alpha_m(x,\,\rho)L^\alpha_n(x,\,\rho)\,\mathrm{d}x=\rho^{m+n+\alpha+1}\frac{\Gamma(\alpha+n+1)}{n!}\delta_{nm}.
$
     \item[\textup{4)}] Integral representation:
     \begin{equation}\label{lgr}
  L_n^{\alpha}(x,\,\rho)=\frac{ e^{\frac{x}{\rho}} x^{-\alpha/2}}{ n!\rho }\int_{0}^{\infty} t^{n+\frac{1}{2}\alpha}\mathrm{J}_{\alpha}(\frac{2}{\rho}\sqrt{xt})e^{-\frac{t}{\rho}} \,\dif t,\quad \alpha>-1,\, n=0,1,2,\cdots,
\end{equation}
where $\mathrm{J}_{v}(x)$ is the Bessel function of order $v$.
     \end{itemize}
 \end{prop}

\subsection{Proof of Theorem~\ref{dingl2}}\label{ssec5.1}
\noindent{\it Proof of Proposition~\ref{pop1}.\,}
It follows form the recursion relation of Hermite polynomial (see Proposition~\ref{lemm0}) that
\begin{align*}
&  \frac{\mi}{\Omega}\calJ [H_k(x)H_{l-k}(y)]
= y H_{l-k}(y)\frac{\dif}{\dif x}H_k(x)-x H_k(x)\frac{\dif}{\dif y}H_{l-k}(y)\\
&= y H_{l-k}(y)k H_{k-1}(x)-x H_k(x)(l-k)H_{l-k-1}(y)\\
&= k H_{k-1}(x)H_{l-k+1}(y)-(l-k)H_{k+1}(x)H_{l-k-1}(y).
\end{align*}
Since $ \frac{\mi}{\Omega} \calJ J_{m,n}(x,y)= \lambda  J_{m,n}(x,y)$ and $J_{m,n}(x,y)=\sum\limits^l_{k=0}\beta_k H_k(x)H_{l-k}(y)$, we obtain that
\begin{eqnarray*}
  \frac{\mi}{\Omega} \calJ [\sum\limits^l_{k=0}\beta_k H_k(x)H_{l-k}(y)]
&=& \sum\limits^l_{k=0} \beta_k [k H_{k-1}(x)H_{l-k+1}(y)-(l-k)H_{k+1}(x)H_{l-k-1}(y)]\\
&=&  \lambda \sum\limits^l_{k=0}\beta_k H_k(x)H_{l-k}(y).
\end{eqnarray*}
Let $\beta_{-1}=\beta_{l+1}=0$. Since $H_k(x,\rho)H_{m-k}(y,\rho)$ are orthogonal, we have that
\begin{equation}\label{eq84}
-(l-(k-1))\beta_{k-1}-\lambda\beta_{k}+(k+1)\beta_{k+1}=0.
\end{equation}
It is exact the linear equation (\ref{linea}).

It follows from Problem~373 and Problem~399 of Reference~\cite{pro} that
\begin{align*}
\det\set{\tensor{M}(\lambda)}  &= \left|
\begin{array}{llllllll}
     -\lambda & \mi &0    & 0   & \dots  & 0 & 0& 0\\
     l\mi & -\lambda &2\mi  & 0   & \dots  & 0 & 0& 0\\
     0   & (l-1)\mi & -\lambda &3 \mi & \dots  & 0 & 0& 0\\
     \hdotsfor{8}\\
     0   & 0   & 0   &0    & \dots  &2\mi &-\lambda  & l\mi\\
     0   & 0   & 0   &0    & \dots  &0 &\mi  & -\lambda
\end{array}
\right| \\
&= \mi^{l+1}  \left|
\begin{array}{llllllll}
     \lambda\, \mi & 1 &0    & 0   & \dots  & 0 & 0& 0\\
     l & \lambda\, \mi &2  & 0   & \dots  & 0 & 0& 0\\
     0   & l-1 & \lambda\, \mi &3  & \dots  & 0 & 0& 0\\
     \hdotsfor{8}\\
     0   & 0   & 0   &0    & \dots  &2 &\lambda\, \mi  & l\\
     0   & 0   & 0   &0    & \dots  &0 &1  & \lambda\, \mi
\end{array}
\right| \\
&= \mi^{l+1}\prod_{m=0}^{l}(\lambda\, \mi +l-2m) =(-1)^{l+1}\prod_{m=0}^{l}(\lambda+ (2m-l)\mi)\\
&=(-1)^{l+1}\prod_{m=0}^{l}(\lambda+ (m-n)\mi).
\end{align*}
Thus when $\lambda=- (m-n)\mi$, Eq.(\ref{linea}) has a nonzero solution\footnote{In fact, we present an alternative way to get Eq.(\ref{spectm}), i.e., the spectrum of the operator $A$.}.

Set $I(\theta)=\sum\limits^l_{k=0}\beta_{k} \cos ^{k}\theta \sin^{l-k}\theta$. Differentiating $I(\theta)$ yields that
\begin{eqnarray*}
  I'(\theta)
  &=&\sum\limits^l_{k=0}\beta_{k} [-k\cos ^{k-1}\theta \sin^{l-k+1}\theta +(l-k)\cos ^{k+1}\theta \sin^{l-k-1}\theta]\\
  &=& - \sum\limits^l_{k=0} [(k+1)\beta_{k+1} -(l-(k-1))\beta_{k-1}] \cos ^{k}\theta \sin^{l-k}\theta\\
  &=&(m-n)\mi\sum\limits^l_{k=0}\beta_{k} \cos ^{k}\theta \sin^{l-k}\theta\quad (\text{by Eq.(\ref{eq84})})\\
  &=&(m-n) \mi I(\theta).
\end{eqnarray*}
Then we have that
\begin{equation}\label{liying}
  I(\theta)=\sum\limits^l_{k=0}\beta_{k} \cos ^{k}\theta \sin^{l-k}\theta=c e^{\mi (m-n)\theta},
\end{equation}
where $c$ is a constant.
Eq.(\ref{liying}) is equivalent to Eq.(\ref{itheta2}) by means of polar coordinate transformation.
{\hfill\large{$\Box$}}

\noindent{\it Proof of Theorem \ref{dingl2}.\,}
We need only to show the case $m\ge n.$
 \begin{align*}
   J_{m,n}(x,y)
    &= \sum\limits^l_{k=0}\beta_{k} H_k(x,\frac{\rho}{2})H_{m-k}(y,\frac{\rho}{2})\\
    &=\sum\limits^l_{k=0}\beta_{k} [ \frac{(-\mi)^l}{\pi\rho}\iint_{\Rnum^2} t^k  s^{l-k}e^{-\frac{1}{\rho}(t-\mi x)^2} e^{-\frac{1}{\rho}(s-\mi y)^2}\,\dif t \dif s ]\quad (\text{by ( \ref{hnin})})\\
    &= \frac{(-\mi)^l}{\pi\rho}c  \iint_{\Rnum^2} e^{-\frac{1}{\rho}(t^2+s^2)+ \frac{2\mi}{\rho} (tx+sy)+\frac{1}{\rho}(x^2+y^2)} (t+\mi s)^{m-n}(t^2+s^2)^n\, \dif t \dif s \quad (\text{by (\ref{itheta2})})\\
    & \quad ( \text{let $c=1,\,t=r\cos \theta,\,s=r\sin\theta,\,\sin\alpha=x/\sqrt{x^2+y^2},\,\cos\alpha = -y/\sqrt{x^2+y^2}$} ) \\
    &= \frac{(-\mi)^l}{\pi\rho} e^{\frac{1}{\rho} (x^2+y^2)} \int^{\infty}_0 e^{-\frac{1}{\rho} r^2}r^{l} r\,\dif r \int^{\pi}_{-\pi} e^{\mi \frac{2r}{\rho} \sqrt{x^2+y^2}\sin(\alpha-\theta)}e^{\mi(m-n)\theta} \,\dif \theta\\
    & \quad ( \text{let $ \gamma={r^2},\, \tau =\alpha-\theta$})\\
    &= (-\mi)^l  e^{\mi(m-n)\alpha }\frac{e^{\frac{1}{\rho} (x^2+y^2)}}{\rho} \int^{\infty}_0 e^{-\frac{\gamma}{\rho}}\gamma ^{l/2} \,\dif \gamma \frac{1}{2\pi}\int^{\pi}_{-\pi}e^{-\mi[(m-n)\tau - \frac{2}{\rho} \sqrt{(x^2+y^2)\gamma}\sin\tau] } \,\dif \tau\\
    &\quad (\text{by the periodicity of triangle function})\\
    &= (-\mi)^l e^{\mi(m-n)\alpha }\frac{e^{\frac{1}{\rho} (x^2+y^2)}}{\rho}  \int^{\infty}_0 \gamma ^{n+\frac{m-n}{2}}\mathrm{J}_{m-n} \big(\frac{2}{\rho} \sqrt{(x^2+y^2)\gamma}\big) e^{-\frac{\gamma}{\rho}}\,\dif \gamma\quad (\text{by (\ref{bess})})\\
    &= (-\mi)^l e^{\mi(m-n)\alpha } n!({x^2+y^2})^{\frac{m-n}{2}} L^{m-n}_{n}(x^2+y^2,\,\rho) \quad (\text{by (\ref{lgr})})\\
    &= (-1)^n  n! (x+\mi y)^{m-n}L^{m-n}_{n}(x^2+y^2,\,\rho).
  \end{align*}
{\hfill\large{$\Box$}}
\subsection{Proof of Theorem~\ref{thjl}}\label{ssec5.2}
\begin{prop}
  \begin{align}
    \partial J_{m,n}&=m J_{m-1,n},\quad  \bar{\partial} J_{m,n}=n J_{m,n-1}.\label{pt}
  \end{align}
\end{prop}
\begin{proof}
Since $\overline{\partial \phi}=\bar{\partial}\bar{\phi}$, we have $ J_{m,n}(z,\,\rho)=\overline{J_{n,m}(z,\,\rho)}$. Thus we need only to prove that
\begin{equation}\label{ee1}
  \partial J_{m,n}=m J_{m-1,n}.
\end{equation}
When $m=0$, both sides of Eq.(\ref{ee1}) equal to $0$. We proceed by induction on $m$.  Assume that Eq.(\ref{ee1}) is true for $n<k$. Then it follows from
the commutation relation $  \partial  \partial^* -  \partial^*\partial=\frac{1}{\rho}$
that
\begin{align*}
   \partial J_{k,n}&=\partial (\rho \partial^*J_{k-1,n})=\rho[\partial^* \partial J_{k-1,n}+\frac{1}{\rho} J_{k-1,n}]\\
   &= \rho \partial^* (k-1) J_{k-2,n} + J_{k-1,n}= (k-1)J_{k-1,n}+J_{k-1,n}=k J_{k-1,n}.
\end{align*}
\end{proof}

\noindent{\it Proof of Theorem~\ref{ijl}.\,} We need only prove that
 \begin{equation}\label{jl2}
    J_{m,n}(z,\,\rho)=z^{m-n}(-1)^nn!L^{m-n}_n (\abs{z}^2,\,\rho), \quad m\ge n.
 \end{equation}
 The case $0= n\le m$ is obvious. We proceed by induction on $m,n$. Assume that Eq.(\ref{jl2}) is true for $l \le m <k$.

 Let $u=\abs{z}^2$. Since $\partial^* $ commutes with $\bar{ \partial}^*$, by the definition Eq.(\ref{itldefn}), we have that
  \begin{align*}
    J_{m,l+1}(z,\,\rho)&=\bar{\partial}^*J_{m,l}(z,\,\rho)=\bar{z}J_{m,l}(z,\,\rho)-\rho{\partial} J_{m,l}(z,\,\rho)\\
    &=\bar{z}J_{m,l}(z,\,\rho)-m\rho J_{m-1,l}(z,\,\rho)  \qquad \text{(by Eq.(\ref{pt}))}\\
    &=(-1)^l l!z^{m-l-1}(uL^{m-l}_l (u,\,\rho)-m\rho L^{m-l-1}_{l} (\abs{z}^2,\,\rho))\\
    &=(-1)^{l+1} (l+1)!z^{m-l-1}L^{m-l-1}_{l+1} (\abs{z}^2,\,\rho).\quad \text{(by  Eq.(\ref{recur2}))}
 \end{align*}
 \begin{align*}
    J_{k+1,n}(z,\,\rho)&=\partial^*J_{k,n}(z,\,\rho)=zJ_{k,n}(z,\,\rho)-\rho\bar{\partial} J_{k,n}(z,\,\rho)\\
    &=zJ_{k,n}(z,\,\rho)-n \rho J_{k,n-1}(z,\,\rho) \qquad \text{(by Eq.(\ref{pt}))}\\
    &=(-1)^n n!z^{k+1-n}(L^{k-n}_n (\abs{z}^2,\,\rho)-\rho L^{k-n+1}_{n-1} (\abs{z}^2,\,\rho))\\
    &=(-1)^n n!z^{k+1-n}L^{k+1-n}_n (\abs{z}^2,\,\rho).\qquad \text{(by Eq.(\ref{recur1}))}
 \end{align*}
{\hfill\large{$\Box$}}

\noindent{\it Proof of Theorem~\ref{thjl}.\,} Denote by $\innp{f,g}$ be the inner product in the complex Hilbert space $L_{\Cnum}^2(\mu)$.
 \begin{itemize}
 \item[\textup{1)}]
If $m-n\ge k-l\ge 0$ then
\begin{align*}
  \innp{J_{m,n},\,J_{k,l}}&= (-1)^{n+l}\frac{n!l!}{\pi \rho}\iint_{\Rnum^2} {z}^{m-n}\bar{z}^{k-l}L_n^{m-n}(\abs{z}^2,\,\rho)L_l^{k-l}(\abs{z}^2,\,\rho)e^{-\frac{x^2+y^2}{\rho}}\dif x\dif y\\
  & \quad ( \text{let } x=r\cos \theta,\, y=r\sin \theta \text{  and  }  s=r^2=x^2+y^2.)\\
  &= (-1)^{n+l}\frac{n!l!}{2\pi \rho}\int_0^{2\pi}e^{\mi [m-n-(k-l)[\theta}\dif \theta \int_0^{\infty} s^{m-n}e^{-\frac{s}{\rho} }L_n^{m-n}(s,\,\rho)L_l^{k-l}(s,\,\rho)\dif s\\
  &= (-1)^{n+l}\frac{n!l!}{\rho}\delta_{m-n,k-l}\int_0^{\infty} s^{m-n}e^{-\frac{s}{\rho} }L_n^{m-n}(s,\,\rho)L_l^{m-n}(s,\,\rho)\dif s\\
  &= {m!n!} \rho^{m+n}\delta_{nl}\delta_{m-n,k-l} \qquad \text{(by 3) of Proposition~\ref{lagu})}.
\end{align*}
The other cases are similar.

This proves that the collection $\set{(m!n!\rho^{m+n})^{-\frac12}J_{m,n}(z,\,\rho)}$ is an orthonomal system. Therefore $\set{J_{m,n}}$ are linearly independent. It follows from Eq.(\ref{jal}) that they generate by linear combination the complex vector space of polynomials in terms of $z,\,\bar{z}$ (equal to in terms of $x,y\in \Rnum$). It follows from Lemma~2.4 of \cite[P6]{Ma} that the polynomials in the coordinate functions are dense in $L^2_{\Cnum}(\mu)$ (Note that the Lemma is still valid for the complex polynomials and the complex Hilbert space). This proves the  orthonomal system is complete.

\item[\textup{2)}]Let $\alpha=m-n,\,u=\abs{z}^2$.
 It follows from Eq.(\ref{pt}) that
  \begin{equation}\label{lapal}
  \frac{\partial^2}{\partial z\partial \bar{z}}J_{m,n}(z,\,\rho)=mnJ_{m-1,\,n-1},
\end{equation}
and if $m\ge n$ then
\begin{align}
  & [(1+\mi c)z \frac{\partial}{\partial z}+(1-\mi c)\bar{z} \frac{\partial}{\partial \bar{z}}-2\rho \frac{\partial^2}{\partial z\partial \bar{z}} ]J_{m,n}(z,\,\rho)\nonumber\\
  &=(-1)^n n! z^{m-n} [(1+\mi c)mL^{\alpha-1}_n  -(1-\mi c)uL^{\alpha+1}_{n-1} + 2\rho m L^{\alpha}_{n-1}]\nonumber\\
  &=(-1)^n n! z^{m-n}\big[(2 m\rho L^{\alpha}_{n-1}+mL^{\alpha-1}_n+u \frac{\partial}{\partial u }L^{\alpha}_{n})+\mi (mL^{\alpha-1}_n-u \frac{\partial}{\partial u }L^{\alpha}_{n})c\big].\label{eig2}
\end{align}
Note that
\begin{align*}
  mL^{\alpha-1}_n-u \frac{\partial}{\partial u }L^{\alpha}_{n}&=mL^{\alpha-1}_n-n L_n^{\alpha}+m\rho L_{n-1}^{\alpha}\\
  &=m(L^{\alpha-1}_n+\rho L_{n-1}^{\alpha})-n L_n^{\alpha}
  =(m-n)L_n^{\alpha},
  \end{align*}
  and
\begin{align*}
   2 m\rho L^{\alpha}_{n-1}+mL^{\alpha-1}_n+u \frac{\partial}{\partial u }L^{\alpha}_{n}
  &=2 m\rho L^{\alpha}_{n-1}+mL^{\alpha-1}_n+ n L_n^{\alpha}- m\rho L_{n-1}^{\alpha}\\
  &=m(\rho  L_{n-1}^{\alpha}+L^{\alpha-1}_n)+n L_n^{\alpha}=(m+n)L_n^{\alpha}.
  \end{align*}
Substituting the above two equations into Eq.(\ref{eig2}) yields  Eq.(\ref{eigen}).
\footnote{There is a direct way to prove Eq.(\ref{eigen}) by using Eq.(E) of \cite[Theorem12]{ito}. In order to be self-contained, here we use the equality of Laguerre polynomials.}

\item[\textup{3)}] For the fixed $\lambda\in \Cnum$, obviously the function $w(z)=\exp\set{\lambda \bar{z} + \bar{\lambda}z-\rho |\lambda|^2}\in L^2(\mu)$. Thus $w(z)$ has a unique series expression
  \begin{equation}\label{express}
    w(z)=\sum_{m=0}^{\infty}\sum_{n=0}^{\infty}a_{m,n}\frac{J_{m,n}(z,\rho)}{m!n!\rho^{m+n}},
  \end{equation}
  where the coefficients $a_{m,n}$ are given by
  \begin{align*}
    a_{m,n}&=\innp{w,\,J_{m,n}}
    =\rho^{m+n}\innp{w,\,(\partial^*)^m(\bar{\partial}^*)^n 1}\\
    &= \rho^{m+n}\innp{\partial^m\bar{\partial}^n w,\, 1}=\bar{\lambda}^m\lambda^n \rho^{m+n}\innp{ w,\, 1}\\
    &=\bar{\lambda}^m\lambda^n \rho^{m+n}.
  \end{align*}
  Substituting it into Eq.(\ref{express}) yields Eq.(\ref{gene}).
  In addition, by the well-known classical global uniform estimates given by Szeg$\ddot{o}$ \cite{sg}: $|L^{\alpha}_n(x)| \le \frac{(\alpha+1)_n}{n!}e^{\frac{x}{2}},\,\,\alpha,x\ge 0$, one can show that the convergence is absolutely and uniform on compact sets in $(\lambda,z)$. Thus the convergence is also pointwise and the equality holds everywhere. \footnote{K.Ito showed Eq.(\ref{gene}) by means of power series expression \cite{ito}. }
\end{itemize}
\subsection{Ito's Complex Multiple Wiener Integral}\label{ssec5.3}
For the reader's convenience, we summarize Ito's work on complex multiple Wiener integral \cite{ito}.
By analogy with the relation between the real multiple Wiener-Ito integral and Hermite polynomials
\begin{equation}\label{1.1}
\begin{split}
  H_n(\xi_t)&=\int_0^t\int _0^{t}\cdots\int_0^{t}\,\dif {\xi}_{t_1}\cdots\dif {\xi}_{t_{n}},
\end{split}
\end{equation}
where $H_n$ denotes the $n$-th Hermite polynomial with leading coefficient 1,
K.Ito obtained the relation between Hermite-Laguerre-Ito polynomials and the complex multiple Wiener-Ito integral.

If $(B_1,B_2)$ denotes 2-dimensional Brownian motion we put $\zeta_t:=B_1(t)+\mi B_2(t)$ with $\mi=\sqrt{-1}$. $\zeta_t$ is called complex Brownian motion. Let  $\bar{\zeta}_t$ be the complex conjugate of $\zeta_t$.
\begin{nott}
  For $m,n\in \Nnum$, denote $F_{m,n}(\zeta_t)=\frac{(-1)^{m \wedge n}}{m!n!}J_{m,n}(\zeta_t,\Enum{\abs{\zeta}_t}^2) $.
\end{nott}

By the formula for integration by parts (stochastic product rule) and Ito's formula,
\begin{prop}\label{pp3}
$F_{m,n}(\zeta_t),\, m,n\in \Nnum$ satisfy that
\begin{equation}\label{digui}
  \dif F_{m,n}(\zeta_t)=F_{m-1,n}(\zeta_t)\dif \zeta_t+F_{m,n-1}(\zeta_t)\dif \bar{\zeta}_t.
\end{equation}
\end{prop}
By iteration,
\begin{cor}
$F_{m,n}(\zeta_t)$ can be decomposed into the iterated Ito integrals of complex Brownian motion as
  \begin{equation}\label{tens2}
    F_{m,n}(\zeta_t)=\sum \int_0^t\int _0^{t_{m+n}}\cdots\int_0^{t_2}\,\dif C_{t_1}\dif C_{t_2}\cdots \dif C_{t_{m+n}},
  \end{equation}
  where $0<t_1<t_2<\cdots <t_{m+n}<t$, $C_t=\zeta_t$ or $C_t=\bar{\zeta}_t$, and the sum is over all choose of $n$ positions of $\set{1,2,\dots,m+n}$ such that $C_t=\bar{\zeta}_t$ .
\end{cor}

Using the approximation by off-diagonal step functions (i.e., the analogue of multiple Wiener-Ito integral \cite[Definition 9.6.5]{guo}),
\begin{eqnarray*}
  &&\int_0^t\int _0^{t}\cdots\int_0^{t}\,\dif \zeta_{t_1}\cdots\dif \zeta_{t_{m-j}} \dif \bar{\zeta}_{t_{m-j+1}} \cdots \dif\bar{\zeta}_{t_m}\\
  &=&  \sum \int_0^t\int _0^{t_m}\cdots\int_0^{t_2}\,\dif C_{t_1}\dif C_{t_2}\cdots \dif C_{t_m},
\end{eqnarray*}
which can be looked as a generalization of \cite[Theorem 9.6.7]{guo} to complex Brownian motion.
Thus,
\begin{cor}\label{cor5}
  $J_{m,n}(\zeta_t,\,\Enum{\abs{\zeta_t}}^2)$ is related to the complex multiple Wiener-Ito integral,
  \begin{equation*}\label{multiintg}
    J_{m,n}(\zeta_t,\,\Enum{\abs{\zeta_t}}^2)= (-1)^{j\wedge (m-j)}j!(m-j)!\int_0^t\int _0^{t}\cdots\int_0^{t}\,\dif \zeta_{t_1}\cdots\dif \zeta_{t_{m-j}} \dif \bar{\zeta}_{t_{m-j+1}} \cdots \dif\bar{\zeta}_{t_m}.
  \end{equation*}
\end{cor}
{\hfill\large{$\Box$}}

\vskip 0.2cm {\small {\bf  Acknowledgements}\   This work was
partly supported by  NSFC(No.11071008, No.11101137.)


\end{document}